\documentclass[a4paper]{amsart}
\usepackage{graphicx}
\usepackage{bbm, caption, color, enumitem, fourier}
\usepackage{pdflscape}
\usepackage{tikz}
\usetikzlibrary{calc, intersections, positioning}

\def\fdd{\stackrel{\textnormal{fdd}}{\Rightarrow}}

\def\d{\mathrm{d}}

\def\C{{\mathbb{C}}}

\def\E{{\mathbb{E}}}

\def\H{{\mathbb{H}}}

\def\N{{\mathbb{N}}}
\def\P{{\mathbb{P}}}

\def\R{{\mathbb{R}}}
\def\S{{\mathbb{S}}}

\def\Z{{\mathbb{Z}}}

\def\cF{{\mathcal F}}

\def\cM{{\mathcal M}}

\def\cP{{\mathcal P}}

\def\cZ{{\mathcal Z}}

\theoremstyle{plain}
\newtheorem{lemma}{Lemma}[section]
\newtheorem{theorem}[lemma]{Theorem}

\newtheorem{corollary}[lemma]{Corollary}

\newtheorem{innerassumption}{Assumption}
\newenvironment{assumption}[1]
{\renewcommand\theinnerassumption{#1}\innerassumption}
{\endinnerassumption}

\theoremstyle{remark}
\newtheorem{remark}[lemma]{Remark}

\title[Height and contour processes of Crump-Mode-Jagers forests (III)]{Height and contour processes of Crump-Mode-Jagers forests (III): The binary, homogeneous universality class} 

\author{Emmanuel Schertzer}
\address{Faculty of Mathematics, University of Vienna, Oskar-Morgenstern-Platz 1, 1090 Wien, Austria}
\email{emmanuel.schertzer@univie.ac.at}

\author{Florian Simatos}
\address{ISAE SUPAERO and Universit\'e de Toulouse\\10 avenue Edouard Belin\\31055 Toulouse Cedex 4\\France}
\email{florian.simatos@isae.fr}

\date{\today}

\numberwithin{equation}{section}

\begin{document}

\begin{abstract}
	This paper belongs to a series of papers aiming to investigate scaling limits of Crump--Mode--Jagers (CMJ) trees. In the previous two papers we identified general conditions under which CMJ trees belong to the universality class of Galton--Watson and Bellman--Harris processes. In this paper we identify general conditions for CMJ trees to belong to the universality class of binary, homogeneous CMJ trees. These conditions state that the offspring process should `look like' a renewal process, and also that it should not accumulate too many atoms near the origin. We show in particular that any renewal process satisfies these conditions
\end{abstract}

\maketitle

\section{Introduction}

This paper is the third paper in a series of papers~\cite{Schertzer18:0, Schertzer19:0} whose goal is to investigate scaling limits of Crump--Mode--Jagers (CMJ) trees, i.e., the chronological trees describing the genealogy of CMJ branching processes. A CMJ branching process is characterized by a pair $(V, \cP)$ with $V$ a positive random variable and $\cP$ a random, finite measure on $(0,\infty)$. Each individual $u$ in the population is endowed with an independent copy $(V_u, \cP_u)$ of $(V, \cP)$ such that:
\begin{itemize}
	\item atoms of $\cP_u$ represent the age of $u$ at childbearing, so that the mass $\lvert \cP_u \rvert$ of $\cP_u$ is the total number of children of $u$;
	\item $V_u$ represents the life-length of $u$ and satisfies $\cP_u((V_u, \infty)) = 0$, i.e., individuals produce their offspring during their life-time.
\end{itemize}

\subsection{Three classes}

Three important classes of CMJ processes are particularly relevant for the present discussion:
\begin{description}
	\item [Galton--Watson (GW) branching processes] they correspond to the case $V = 1$ and $\cP = \lvert \cP \rvert \delta_1$, i.e., all offspring occurs at time $1$ (here and throughout $\delta_x$ denotes the Dirac measure at $x \in \R$);
	\item [Bellman--Harris (BH) branching processes] they correspond to the case where $V$ and $\cP$ are independent and $\cP = \lvert \cP \rvert \delta_V$, i.e., offspring is independent of life length and occurs at death;
	\item [Binary, homogeneous branching processes] they correspond to the case where $\cP = \mu \mid _V$ with $\mu$ a Poisson process independent from $V$ (here and throughout, for a measure $\pi$ on $(0,\infty)$ and $v > 0$, $\pi \mid_v$ denotes the restriction of $\pi$ to $(0,v]$).
\end{description}

Scaling limits of the random chronological trees corresponding to these branching processes, as encoded by their height and contour processes (see in~\cite{Schertzer18:0, Schertzer19:0} and below for definitions), are well-known. The first results were obtained for GW processes with finite variance offspring distribution, whose scaling limit was shown to be the Brownian tree~\cite{Aldous93:0, Bennies00:0, Marckert03:0}. A striking feature in this case, the limiting Brownian excursion coding the tree and the depth-first exploration process (or Lukasiewicz path, see Remark~\ref{rk:SSC}) coincide: as will be seen, this is a feature that we will also encounter in the present paper. Later, the finite variance assumption was relaxed~\cite{Duquesne02:0, Duquesne03:0}, leading to a complete picture in the GW case: scaling limits belong to the class of L\'evy trees.

Seeing a BH process as a GW process where edges have been stretched by i.i.d.\ positive weights, it is natural to describe the scaling limits of BH processes in the same way starting from the L\'evy tree of the underlying genealogical GW process. Performing this construction on the limiting continuous object amounts to considering the so-called L\'evy snake, obtained from a L\'evy tree by adding Poissonian marks, and indeed it was shown that scaling limits of BH processes are described by L\'evy snakes~\cite{Janson05:0, Marckert03:1}.

Finally, the binary, homogeneous case exhibits the striking feature that, if properly defined, its contour process is a spectrally positive L\'evy process~\cite{Lambert10:0}. This observation has far-reaching consequences, but in particular, it immediately entails that the scaling limit of the chronological tree is a continuous tree coded by a spectrally positive L\'evy process.

\subsection{The GW and BH universality classes}

Our goal in this on-going project is to describe scaling limits of CMJ trees. In our earlier work we identified general conditions guaranteeing that CMJ processes belong to the universality class of GW or BH processes.

Namely, we showed in~\cite{Schertzer18:0} that CMJ with `short' edges belong to the GW universality class. Intuitively, this is not surprising since if edges are `short' in a well-defined sense, then it does not really matter when births occur precisely, and the chronological tree should look like the underlying genealogical tree. We identified in~\cite{Schertzer18:0} a precise `short edge' condition in the form of finite moments $\E(R), \E(V) < \infty$. Here $R$ represents the ``typical'' age of an individual at child-bearing, and its mean is simply given by $\E(R) = \E(\int_0^\infty u \cP(\d u))$.

In~\cite{Schertzer19:0} we showed a more subtle result: we allowed for edges to be long and considered the case where $\E(R) = \E(V) = \infty$, but we imposed that individuals who live for a long time do not have a large offspring. This is what happens in the BH case, and indeed we identified a general condition under which scaling limits of CMJ processes belong to the BH universality class. Moreover, we found the very surprising result (at least, to us) that any CMJ process with finite variance offspring distribution satisfies this condition, and thus belongs to the BH universality class.

We will also make use of general convergence results established in \cite{Schertzer18:0, Schertzer19:0}, which make it possible to reduce the joint convergence of the Lukasiewicz path, height and contour processes to the study of a multi-dimensional renewal process constructed from the ladder height process associated to the Lukasiewicz path , see the discussion following Theorem~\ref{thm:main}.

\subsection{The binary, homogeneous universality class}

Consider now the binary, homogenous case. Here the situation is opposite to the BH case: because individuals give birth at constant rate during their lifetime, individuals who live for a long time have a large offspring.

In the present paper we consider the case $\E(R) = \infty$, i.e., the ``typical'' age at child-bearing is large, but $\E(V) < \infty$, which is intermediate between the two-aforementioned cases (see Remark~\ref{rk:VR}). Our goal is to identify a large class of CMJ processes that have the same scaling limit as binary, homogeneous CMJ processes. Heuristically, we consider a case similar to the binary, homogeneous case where $\cP = \mu \mid_V$ with $\mu$ and $V$ independent, and we identify precise conditions for $\mu$ to `look like' a renewal process so that the CMJ process has the same scaling limit as the corresponding binary, homogeneous CMJ process.

In the present paper we only treat the non-triangular case. For the GW and BH universality classes, non-triangular results readily extend to the triangular case. However, here we found that the situation is more subtle and that we may have $\mu$ which `looks like' a renewal process and yet, the scaling limit be different from the scaling limit of the corresponding binary, homogeneous CMJ process.

\subsection{Organization of the paper} Section~\ref{sec:statement} introduces notation and states the main result of the paper, and Section~\ref{sec:models} applies this result to two cases. Section~\ref{sec:proof} presents the proof of the main result.

\section{Statement of main results} \label{sec:statement}

\subsection{General notation}

Let $\N$ the set of non-negative integers. For $x, y \geq 0$ let $[x] = \max\{n \in \Z: n \leq x\}$. For $d \in \N$ let $D(\R^d)$ be the set of c\`adl\`ag functions from $\R$ to $\R^d$, endowed with the Skorohod topology. For $f \in D(\R)$ we denote by $f(t-)$ its left-limit at $t \in \R_+$ and by $\Delta f(t) = f(t) - f(t-)$ the size of its jump.

We let $\cM$ be the set of positive Radon measures on $(0,\infty)$ and $\delta_x \in \cM$ for $x \geq 0$ be the Dirac measure at $x$. The mass of $\nu \in \cM$ will be denoted by $\lvert \nu \rvert = \nu((0,\infty))$, and $\nu \mid_x$ denotes the measure $\nu$ stopped at $x$: $\nu\mid_x(A) = \nu(A \cap [0,x])$ for any Borel set $A \subset (0,\infty)$. For convenience we will write $\nu(t) = \nu((0,t])$ for any $\nu \in \cM$ and $t \geq 0$, and thus make the usual identification between non-decreasing c\`adl\`ag functions starting at $0$ and positive measures on $\R_+$. We define $A_\nu(n)$ as the location of the $n$th atom of $\nu$:
\[ A_\nu(n) = \min \left\{ t \geq 0: \nu(t) \geq n \right\}, \ n \in \N, \]
with the convention $\min \emptyset = \infty$.

We will use $\Rightarrow$ to denote weak convergence, in the Skorohod topology when considering processes, and $\fdd$ to denote convergence of finite-dimensional distributions: for processes $X, X_p: \R \to \R^d$ we have $X_p \fdd X$ if and only if for every finite subset $I \subset \R_+$ we have $(X_p(t), t \in I) \Rightarrow (X(t), t \in I)$.

\subsection{CMJ process}

We consider a probability space $(\Omega, \cF, \P)$ rich enough to consider all necessary random variables. We consider a random variable $V > 0$ and a random measure $\mu \in \cM$ which are independent under $\P$. The process $A_\mu$ will often be encountered, and we will simply denote it by $A$, i.e., we define $A = A_\mu$. Note that $A(n) > 0$ for $n \geq 1$.

Throughout the paper, we consider the CMJ process with characteristic $(V, \cP)$ with $\cP = \mu \mid _V$, so that $\lvert \cP \rvert = \mu(V)$.

Letting $\cP_k$ denote the measure according to which the $k$th individual (ranked in lexicographic order) gives birth, we consider the Lukasiewicz path $S = (S(n), n \geq 0)$ with $S(0) = 0$ and
\[ S(k) = \sum_{k=0}^{n-1} \left( \lvert \cP_k \rvert - 1 \right), \ n \geq 1. \]

We finally consider $\H$ and $\C$ the chronological height and contour processes associated to the CMJ forest with characteristic $(V, \cP)$: $\H(n)$ denotes the chronological height, i.e., birth time, of the $n$th individual, while $\C(t)$ is the distance to the root at time $t$ of a particle traveling along edges of the chronological tree at unit speed, see~\cite{Schertzer18:0} for precise definitions and illustrations\footnote{In~\cite{Schertzer19:0} we studied a different chronological contour process: to avoid any ambiguity we stress that in the present paper, we consider the `classical' contour process as defined in~\cite{Schertzer18:0}.}.

\subsection{Assumptions on $(V, \mu)$}

We gather here various assumptions that will be enforced on $(V, \mu)$. Recall that $V$ and $\mu$ are assumed to be independent.

\begin{assumption}{\text{\textnormal{A}}} \label{ass-A}
	First, the usual criticality assumption:
	\begin{itemize}
		\item[\text{\textnormal{(C)}}] \label{ass-C} $\E(\mu(V)) = 1$.
	\end{itemize}
	Second, an assumption on $V$:
	\begin{itemize}
		\item[\text{\textnormal{(V)}}] \label{ass-V} there exists $\gamma \in (1,2)$ such that $V$ is in the domain of attraction of a $\gamma$-stable distribution.
	\end{itemize}
\end{assumption}

\noindent It is well-known that Assumption~V is equivalent the existence of $v_p \to 0$ such that $p \P(V \geq x / v_p) \to x^{-\gamma}$ for every $x > 0$, see for instance the proof of $\S$35, Theorem $2$ in~\cite{Gnedenko68:0}. Actually, this fact will be used several times in the sequel, without further notice. Moreover, $v_p$ is of the form $v_p = p^{-1/\gamma} \ell(p)$ for some slowly varying function~$\ell$ (see for instance~\cite[Section $8.3.2$]{Bingham89:0}). In the sequel, $v_p$ will always refer to such a sequence, provided that Assumption~V is enforced.

The next two conditions enforce that $\mu$ `looks like' a renewal measure, and so the corresponding height and contour processes will have the same scaling limit than that of a binary, homogeneous CMJ process.

\begin{assumption}{\text{\textnormal{A (continued)}}}
	Third, two assumptions on $\mu$:
	\begin{enumerate}
		\item[\text{\textnormal{(R1)}}] \label{ass-R1} there exists a constant $a \in (0,\infty)$ such that $\mu(t) / t \to 1/a$ as $t \to \infty$, where the convergence takes place in $L_1$;
		\item[\text{\textnormal{(R2)}}] \label{ass-R2} $p \P(v_p \mu(x/v_p) /x \geq 1/a') \to 0$ for some $a' < a$ and every $x > 0$.
	\end{enumerate}
\end{assumption}

Assumption R1 says that the number of atoms of $\mu$ grows linearly, which suggests that they are roughly evenly spaced as in a renewal process: it is therefore the assumption that states that $\mu$ looks like a renewal process. The interpretation of Assumption R2 is that $\mu$ cannot accumulate too many atoms near $0$. It may look technical at first, but we provide a simple example in Section~\ref{sub:counter-example} showing that our main result may not hold if something like that is not imposed. We will see in Section~\ref{sec:models} that R2 is automatically satisfied for a renewal process.

\subsection{Main result}

Let in the sequel
\[ S_p(t) = v_p S([pt]), \ t \geq 0, \]
with $v_p$ from Assumption~\ref{ass-A}. Let also $S_\infty$ be the spectrally positive, $\gamma$-stable L\'evy process with Laplace exponent
\begin{equation} \label{eq:psi}
	\psi(\lambda) = \frac{1}{- \Gamma(1-\gamma)} \left( \frac{\lambda}{a} \right)^\gamma, \ \lambda > 0,
\end{equation}
with $\Gamma$ the usual Gamma function. For $p \in [0,\infty]$ let $\underline S_p$ be the path $S_p$ reflected above its past infimum:
\begin{equation}\label{eq:luka-ref} \underline S_p(t) \ = \ S_p(t) - \inf_{u\leq t} S_p(u).\end{equation}
In addition to $S_p$, for finite $p$ we consider $\H_p(t) = v_p \H([pt])$ and $\C_p(t) = v_p \C(pt)$, i.e., we use the same scaling for the Lukasiewicz path and the chronological height and contour processes. Our goal is to prove the following result.

\begin{theorem} \label{thm:main}
	If Assumption~\ref{ass-A} holds, then
	\[ (\H_p, \C_p, \underline S_p) \fdd \ \left( a \underline S_\infty, a \underline S_\infty(\cdot/2\E(V)), \underline S_\infty \right). \]
\end{theorem}

Let us briefly outline the strategy of the brief that builds on previous results established in \cite{Schertzer18:0,Schertzer19:0}.
First, we recall Theorem~$4.1$ in~\cite{Schertzer18:0} which, in the non-triangular case of the present paper, can be stated as follows.

\begin{theorem}[Theorem $4.1$ in~\cite{Schertzer18:0}]
	If the following conditions hold:
	\begin{itemize}
		\item $\E(V) < \infty$ and $\E(\mu(V)) = 1$;
		\item $S_p \Rightarrow S_\infty$;
		\item $\H_p \fdd \H_\infty$ for some process $\H_\infty$ which is (almost surely) continuous at $0$ and satisfies the condition $\P(\H_\infty(t) > 0)$ for every $t > 0$;
	\end{itemize}
	then $(\H_p, \C_p) \fdd (\H_\infty, \H_\infty(\, \cdot \,/(2\E(V)))).$
\end{theorem}

As the two first assumptions of this result hold under Assumption A (see Lemma~\ref{lemma:DA-mu(V)} for the convergence $S_p \Rightarrow S_\infty$), it follows from this result that in order to prove~\eqref{eq:conv-thm} it is enough to prove that $(\H_p, \underline S_p) \fdd \ (a \underline S_\infty, \underline S_\infty)$.
Secondly, we will prove the convergence of $(\H_p(t), \underline S_p(t))$ using an identity in law with a two-dimensional renewal process evaluated at a random time. 
See Theorem \ref{thm:recap-2} for more details.

To summarize, under the conditions stated above, the joint convergence of the Lukashiewicz, the height and the contour processes boil down to prove a scaling limit result on a two-dimensional renewal process.

\begin{remark} \label{rk:SSC}
	Our result states that the height and contour processes coincide asymptotically with the Lukasiewicz path up to a deterministic space-time change. This result is in the same spirit of \cite{Marckert03:0} where it was shown that for critical Galton-Watson planar trees with finite variance $\sigma^2$ conditioned on having $p$ vertices, 
	\[ \left( \H_p, \C_p, S_p \right) \Rightarrow \left( \frac{2}{\sigma} e_\infty, \frac{2}{\sigma} e_\infty(\cdot/2), e_\infty \right) \ \ \mbox{as $p\to\infty$}\]
	where $\H_p, \C_p, S_p$ are the properly scaled height, contour and Lukasiewicz path of the random tree, and $e_\infty$ is the standard Brownian excursion.
\end{remark}

\begin{remark}
	In contrast to Theorem~\ref{thm:main}, the convergence in the previous display holds in a functional sense, whereas our result only holds in the sense of finite-dimensional distributions. In~\cite{Schertzer18:0,Schertzer19:0} we constructed several examples where the chronological height process of a CMJ converges in the sense of finite-dimensional distributions but is not tight, which prevents functional convergence. However, we conjecture that under Assumption~\ref{ass-A}, there is actually functional convergence, and not only convergence of finite-dimensional distributions.
\end{remark}

\subsection{Example showing the necessity of Assumption~R2} \label{sub:counter-example}

As already explained, our goal is to generalize known results on the binary, homogeneous case. Intuitively, we expect a CMJ process where births occur `like' a renewal process to behave similarly: upon scaling, only long edges matter, and so if upon scaling the offspring measure $\cP$ looks like Lebesgue measure, the considered CMJ process should belong to the binary, homogeneous universality class. This is exactly the rational behind Assumption~R1.

However, the following example shows that Assumptions~V,~C and~R1 may be satisfied and yet, $S_p$ does not converge (and thus, the conclusion of Theorem~\ref{thm:main} fails). Thus, additional conditions are called upon, and
the following example shows that Assumption~R2 (which prevents accumulation of atoms near $0$) is sharp since if it is not satisfied then the conclusion may not hold.
\\

Let $\tau$ be an integer-valued random variable in the domain of attraction of a $\gamma'$-stable distribution with $1 < \gamma' < \gamma$: $\P(\tau \geq x) = \ell(x) / x^{\gamma'}$ for some slowly varying function $\ell$. Assuming that $\tau$ and $V$ are independent, we then consider
\[ \mu = \tau \delta_{1/2} + \sum_{k \geq 1} \delta_k. \]
We then have $\mu(t) = [t] + \tau 1_{t \geq 1/2}$ so that~R1 is satisfied.  Choose moreover $V$ independent from $\mu$ such that~V and~C are satisfied, in particular $V$ is in the domain of attraction of a $\gamma$-stable distribution. However, for every $a'<a$ and $x>0$
\[ \liminf_{p \to \infty} p \P(v_p \mu(x/v_p) /x \geq 1/a') \geq\liminf_{p \to \infty} p \P(v_p \tau /x \geq 1/a') = \infty \]
so that~R2 is not satisfied. (Here we use the fact that $\gamma' <\gamma$ so that the tail of $\tau$ is heavier than the tail of $V$.)

Let us now argue that Theorem \ref{thm:main} does not hold. We have $\mu(V) = [V] + \tau 1_{V \geq 1/2}$ and so $\mu(V)$ is roughly the sum of two independent random variables with tails $\propto x^{-\gamma}$ and $\propto x^{-\gamma'}$. In this case it is clear that the random variable with the heavier tail will dominate, which in our case is $\tau$ since $\gamma' < \gamma$. In particular, we have the approximation $\P(\mu(V) \geq n) \approx n^{-\gamma'}$. In this case, we know that $S([p])$ is of the order of $p^{1/{\gamma'}} \gg p^{1/\gamma}$, so that $\lvert S_p(1) \rvert \Rightarrow +\infty$.

\begin{remark}
	Beside providing a counter-example to Theorem~\ref{thm:main} when Assumption~R2 fails, we believe that this case is interesting in its own. We conjecture that in this case, the height and Lukasiewicz do not have the same scaling, and it would be interesting to study its scaling limit.
\end{remark}

%

%
%

\section{Two models} \label{sec:models}

For the sake of illustration, we propose two explicit models for which Assumptions R1 and R2 are also satisfied.
\\

\noindent {\bf Model 1 (renewal).} We first show that our results apply to any renewal process. More precisely, let $(\xi_n, n \geq 1)$ be non-negative i.i.d.\ random variables with finite mean and $\mu = \sum_{n=1}^\infty \delta_{\sum_{k=1}^n \xi_k}$ the renewal process with step distribution $\xi_1$ and no delay.

\begin{lemma}
	If $\mu$ is the aforementioned renewal process, then Assumptions~\textnormal{R1} and~\textnormal{R2} hold with $a = \E(\xi_1)$ and any $a' < a$.
\end{lemma}

\begin{proof}
	Assumption~R1 is a direct consequence of the strong law of large numbers, while~R2 comes from Markov inequality: for any $\theta > 0$,
	\[ p \P(v_p \mu(x/v_p) / x \geq 1 / a') = p \P( A(x / (a'v_p)) \leq x/v_p) \leq p e^{\theta / v_p + \log \E(e^{-\theta \xi}) / (v_p a')} \]
	and since $\E(\xi) = a > a'$, we can choose $\theta$ such that $\theta + \log \E(e^{-\theta \xi}) / a' = \varphi < 0$, so that
	\[ p \P(v_p \mu(x/v_p) / x \geq 1 / a') \leq p e^{-\varphi / v_p}. \]
	It is well known that $1/v_p \geq p^\delta$ for some $\delta > 0$ (see for instance~\cite[Proposition $1.3.6$]{Bingham89:0}), so that $p \P(v_p \mu(x/v_p) / x \geq 1 / a') \to 0$ as desired.
\end{proof}

\noindent {\bf Model 2 (switching Poisson).} We now consider a second model where $\mu$ is a Poisson process but where the intensity evolves randomly. The model we consider illustrates how Assumptions R1 and R2 can be checked, it could probably be generalized in various directions.

Let $(\sigma_n, n \geq 0)$ be a sequence of positive i.i.d.\ real-valued random variables with finite mean $m$ and $\Sigma_k = \sigma_1 + \cdots + \sigma_k$ for $k \geq 0$ with $\Sigma_0 = 0$. Let $L \subset (0,\infty)$ be a finite set and $\Lambda = (\Lambda_n, n \geq 0)$ an $L$-valued irreducible Markov chain with unique invariant distribution $\Lambda_\infty$.

The offspring process $\mu$ that we consider is defined as follows: conditional on $\Lambda$, for every $k \geq 1$, $\mu$ restricted on the interval
$[\Sigma_{k-1}, \Sigma_k)$ is an independent Poisson point process with intensity $\Lambda_{k-1}$.

\begin{lemma}
	If $\mu$ is the aforementioned measure, then Assumptions \textnormal{R1} and~\textnormal{R2} are satisfied with $a = 1/\E(\Lambda_\infty)$ and any $a' < 1/\max L$.
\end{lemma}

\begin{proof}
	Let us first check Assumption R1. For $t \geq 0$ let $N(t) = \min\{n \geq 0: \Sigma_n \geq t\}$: then $t \leq N(t)$ and so
	\[ \frac{\mu(t)}{t} \leq \frac{1}{t} \mu \left( \Sigma_{N(t)} \right) = \frac{N(t)}{t} \times \frac{1}{N(t)} \sum_{k=1}^{N(t)} \mu \left( [\Sigma_{k-1}, \Sigma_k) \right), \]
	The strong law of large numbers implies $N(t) / t \to 1/m$ almost surely. Moreover, $(\mu([\Sigma_{k-1}, \Sigma_k)), k \geq 1)$ is by construction an ergodic Markov chain with ergodic measure a mixed Poisson process with random intensity $\sigma_0 \Lambda_\infty$, the two terms $\sigma_0$ and $\Lambda_\infty$ being independent in this product. Thus the ergodic theorem implies that
	\[ \frac{1}{N(t)} \sum_{k=1}^{N(t)} \mu \left( [\Sigma_{k-1}, \Sigma_k) \right) \to \E \left( \sigma_0 \Lambda_\infty \right) = m \E \left( \Lambda_\infty \right), \]
	the convergence taking place almost surely  and in $L^1$. This implies that $\mu(t) / t \to \E(\Lambda_\infty)$ in $L^1$.

	
	Let us now check Assumption R2. Consider $\mu^*$ a Poisson process with intensity $\ell^* := \max L$. We can couple $\mu$ and $\mu^*$ such that $\mu \leq \mu^*$, so that
	\[ \P \left( \frac{v_p}{x} \mu \left( \frac{x}{v_p} \right) \geq \frac{1}{a'} \right) \leq \P \left( \mu^* \left( \frac{x}{v_p} \right) \geq \frac{x}{a' v_p} \right). \]
	If $X$ is a Poisson random variable with parameter $\kappa$, then for $y > \kappa$ we have the classical large-deviation bound
	\[ \P(X \geq y) \leq \exp \left( -y \varphi \left( \frac{y}{\kappa} \right) \right) \ \text{ with } \varphi(u) = u-1-\log(u) \geq 0. \]
	Since $\mu^*$ is a Poisson process with intensity $\ell^*$, for $a' \ell^* < 1$ this gives
	\[ p \P \left( \frac{v_p}{x} \mu \left( \frac{x}{v_p} \right) \geq \frac{1}{a'} \right) \leq p \exp \left( -\frac{x}{a' v_p} \varphi \left( \frac{1}{a' \ell^*} \right) \right). \]
	Since $a' \ell^* < 1$, we have in particular $\varphi(1/(a' \ell^*)) \neq 0$. Moreover, as already mentioned above we have $1/v_p \geq p^\delta$ for some $\delta > 0$ which implies that the previous upper bound vanishes. This shows that Assumption R2 holds and concludes the proof.
\end{proof}

\section{Proof of Theorem \ref{thm:main}} \label{sec:proof}

We provide the proof of Theorem~\ref{thm:main} in this section: in particular, throughout this section we assume that Assumption A holds, even though this is not repeated in the statements of the forthcoming results.



\subsection{Preliminary results} \label{sub:preliminary}

%
%


\subsubsection{Tail behavior of $\mu(V)$} \label{sub:mu(V)}

Assumption~C is the usual criticality assumption, which makes the Lukasiewicz path a random walk with mean $0$. However, what we need is the genealogy to converge, and for that we need $\mu(V)$ to be in the domain of attraction of a stable law. We now show that under Assumption~A, the scaling of $S$ is governed by that of $V$.

\begin{lemma} \label{lemma:DA-mu(V)}
	If Assumption~\ref{ass-A} holds, then $p \P(\mu(V) \geq x / v_p) \to (ax)^{-\gamma}$ for every $x > 0$. In particular, $S_p \Rightarrow S_\infty$.
\end{lemma}

\begin{proof}
	Let $\overline F(x) = \P(V \geq x)$, so that $v_p$ satisfies by definition $p \overline F(x v_p) \to x^{-\gamma}$.
	
	First a preliminary result: if $x_p \to x$, then $p \overline F(x_p / v_p) \to x^{-\gamma}$. Indeed, let $\varepsilon > 0$ and $p$ large enough so that $(1-\varepsilon) x \leq x_p \leq (1+\varepsilon) x$. Then by monotonicity of $\overline F$,
	\[ p \overline F((1+\varepsilon) x / v_p) \leq p \overline F(x_p / v_p) \leq p \overline F((1-\varepsilon) x / v_p) \]
	and so
	\[ \limsup_p p \overline F(x_p / v_p) \leq ((1-\varepsilon) x)^{-\gamma} \to x^{-\gamma} \]
	with the last limit obtained as $\varepsilon \downarrow 0$. A similar result is obtained for the $\liminf$, which proves the claim.
	
	Next, define $X_p = v_p A(x/v_p)$ and write
	\begin{align*}
		p \P(\mu(V) \geq x / v_p) & = \E \left( p \overline F(A(x / v_p)) \right)\\
		& = \E \left( p \overline F(X_p / v_p) \right)\\
		& = \E \left( p \overline F(X_p / v_p)); X_p \geq a' x \right) + \E \left( p \overline F(X_p / v_p) ; X_p < a' x \right).
	\end{align*}
	Since $X_p \Rightarrow ax$ by Assumption~R1 and $a' < a$, we have
	\[ p \overline F(X_p / v_p) 1(X_p \geq a' x) \Rightarrow (ax)^{-\gamma} \]
	by the preliminary result, and since
	\[ p \overline F(X_p / v_p) 1(X_p \geq a' x) \leq p \overline F(a' x / v_p / 2) \to (a'x)^{-\gamma}, \]
	we get by domination
	\[ \E \left( p \overline F(X_p / v_p)); X_p \geq a' x \right) \to (a x)^{-\gamma}. \]
	As for the second term, we have
	\begin{align*}
		\E \left( p \overline F(X_p / v_p) ; X_p < a' x \right) & \leq p \P(X_p \leq a' x/2)\\
		& = p \P(A(x / v_p) \leq a' x / v_p)\\
		& = p \P(x / v_p \leq \mu(a' x / v_p))
	\end{align*}
	which is assumed to vanish by Assumption~R2.
	
	This proves that $p \P(\mu(V) \geq x / v_p) \to (ax)^{-\gamma}$. In order to prove convergence of the rescaled random walk $S_p$, it is sufficient to prove convergence of $S_p(1)$. We compute its Laplace transform and show that it is given by~\eqref{eq:psi}. We have just proved that $\mu(V)$ is in the domain of attraction of a $\gamma$-stable distribution, and so there exists $\ell$ slowly varying with $\P(\mu(V) \geq x) = \ell(x) / x^\gamma$. In particular, the result that we have just proved says that $p v_p^\gamma \ell(1/v_p) \to a^{-\gamma}$. We then have $\P(\mu(V) - 1 \geq x) = \ell(x+1) / (x+1)^\gamma$: since $\ell(x+1) \sim \ell(x)$ because $\ell$ is slowly varying (as a consequence of the uniform convergence theorem~\cite[Theorem 1.2.1]{Bingham89:0}), Theorem 8.1.6 in~\cite{Bingham89:0} gives
	\[ \E(e^{-s(\mu(V)-1)}) = 1 + s^\gamma \ell'(1/s) + o(s^\gamma \ell'(1/s)) \]
	with $\ell'(x) = \ell(x)/(-\Gamma(1-\gamma))$, as $s \downarrow 0$, and so
	\begin{align*}
		\E(e^{-\lambda S_p(1)}) & = \E(e^{-\lambda v_p (\mu(V)-1)})^p\\
		& = \exp \left( p \log(1 + \lambda^\gamma v_p^\gamma \ell'(1/(\lambda v_p)) + o(v^\gamma_p \ell'(1/v_p))) \right).
	\end{align*}
	We have
	\[ p v^\gamma_p \ell'(1/v_p) = \frac{p v_p^\gamma \ell(1/v_p)}{- \Gamma(1-\gamma)} \to \frac{1}{-a^\gamma \Gamma(1-\gamma)} \]
	so that
	\[ \E(e^{-\lambda S_p(1)}) = \E(e^{-\lambda v_p (\mu(V)-1)})^p \to \exp \left( \lambda^\gamma / (-a^\gamma \Gamma(1-\gamma)) \right) \]
	which gives the result.
\end{proof}

\subsubsection{Ladder height process and genealogical results} \label{sub:genealogical-results}

Let $(T, \cZ)$ be the (ascending) ladder process associated to $S$. The ladder time process $T = (T(n), n \geq 0)$ is defined by $T(0)=0$ and for $k \geq 0$,
\[ T(k+1) = \inf \big\{ \ell > T(k): S(\ell) \geq S(T(k)) \big\}, \]
with the convention $T(k+1) = \infty$ if $T(k) = \infty$, and the ladder height process $\cZ = (\cZ(n), n \geq 0)$ is defined by $\cZ(n) = S(T(n))$. We scale the processes $T$ and $\cZ$ in the following way:
\[ \cZ_p(t) = v_p \cZ([p v_p t]) \ \text{ and } \ T_p(t) = \frac{1}{p} T([p v_p t]). \]
Recall that $S_p \Rightarrow S_\infty$, and let also $(T_\infty, \cZ_\infty)$ be the ladder process associated to the L\'evy process $S_\infty$ introduced earlier. According to~\cite[Theorem VII.4]{Bertoin96:0}, the Laplace exponent of $\cZ_\infty$ is given by $\psi_Z(\lambda) = c \lambda^{\gamma-1}$ for some constant $c > 0$.

\begin{lemma} \label{lemma:S-T-Z}
	As $p \to \infty$ we have $\underline S_p \Rightarrow \underline S_\infty$ and $(T_p, \cZ_p) \Rightarrow (T_\infty, \cZ_\infty)$.
\end{lemma}

\begin{proof}
	The first convergence comes from the continuity of the Skorohod reflection map~\cite[Theorem $13.5.1$]{Whitt02:0}, and the second convergence was proved in~\cite[Proposition $3.1$]{Schertzer19:0} to be a consequence of $S_p \Rightarrow S_\infty$.
\end{proof}

We need to consider a third renewal process introduced in~\cite{Schertzer18:0}. Recall that $\cP_k$ is the offspring measure of the $k$th individual and that $A_\nu(n)$ is the location of the $n$th atom of $\nu$. We define the renewal process $R$ with increments $A_{\cP_{T(k)-1}}(\Delta \cZ(k))$ where $\Delta \cZ(k) = \cZ(k)- \cZ(k-1)$, i.e.,
\[ R(n) = \sum_{k = 1}^n A_{\cP_{T(k)-1}}(\Delta \cZ(k)), \]
and scale it in the following way:
\[ R_p(t) = v_p R([p v_p t]), \ t \geq 0. \]
For $p \in \N \cup \{\infty\}$ let $\overline S_p(t) = \sup_{[0,t]} S_p - S_p(t)$ be the process $S_p$ reflected above its past supremum and $L_p$ its local time process at $0$, which is the right-continuous inverse of $T_p$:
\[ L_p(t) = \inf \left\{ s \geq 0: T_p(s) > t \right\}. \]

The previous definitions are motivated by the following result from one of our previous works, which expresses an important equality in distribution.

\begin{theorem}[Theorem 2.2 in \cite{Schertzer19:0}.]\label{thm:recap-2}
There exists some sequence $\rho_p \to \infty$ such that for each $t \geq 0$ we have\footnote{Some help for the thorough reader: if $L$ is defined as the local time at $0$ for the Lukasiewicz path reflected at its past maximum, then $\widetilde T^{-1}$ in~\cite{Schertzer19:0} is equal to $L-1$.} 
\[ (\H_p(t), \underline S_p(t)) = (R_{\rho_p}, \cZ_{\rho_p}) \big( L_{\rho_p}(t) - \big) \ \ \mbox{in law}. \]
where $\underline S$ is path $S$ reflected above its past infimum as defined in (\ref{eq:luka-ref}). 
\end{theorem}

The latter result entails that the convergence of the height process boils  down to the the convergence of the renewal processes introduced above.
This is the content of the up-coming sections.

\subsubsection{Two lemmas}

In the sequel we let $V_x$ denote a random variable distributed as $V/x$ conditioned on $V \geq x$, and we also consider the function $f(t) = \E(\mu(t) / t)$ for $t > 0$. The following two simple lemmas will be crucial for the next arguments.

\begin{lemma} \label{lemma:V_x}
	As $x\to\infty$ we have $V_x \Longrightarrow V_\infty$ in $L^1$.
\end{lemma}

\begin{proof}
	Since $V$ is in the domain of attraction of a $\gamma$-stable distribution, there exists $\ell$ slowly varying with $\P(V \geq x) = \ell(x) / x^\gamma$. For $y \geq 1$ we have
	\[ \P(V \geq yx \mid V \geq x) = \frac{\P(V \geq yx)}{\P(V \geq x)} = \frac{\ell(xy)/\ell(x)}{y^\gamma} \]
	and so the convergence $V_x \Rightarrow V_\infty$ follows from the fact that $\ell$ is slowly varying. Let us now prove that $V_x$ is uniformly integrable. For $K \geq 1$ we have
	\[ \E(V_x ; V_x \geq K) = \E(V/x 1(V / x \geq K) \mid V \geq x) = \frac{\E(V ; V \geq x K)}{x \P(V \geq x)}. \]
	We can compute
	\[ \E(V ; V \geq x K) = x K \P(V \geq x K) + \int_{x K}^\infty \P(V \geq u) \d u. \]
	Theorem~$1.5.11$ in~\cite{bingham89:0} entails that, as $x \to \infty$,
	\[ \E(V ; V \geq x K) \sim \frac{\gamma}{\gamma-1} x K \P(V \geq x K) \]
	and so
	\[ \E(V_x ; V_x \geq K) \sim \frac{\frac{\gamma}{\gamma-1} x K \P(V \geq x K)}{x \P(V \geq x)} = \frac{\gamma}{\gamma-1} \frac{\ell(xK)/(xK)^{\gamma-1}}{\ell(x)/x^{\gamma-1}} \sim \frac{\gamma}{(\gamma-1) K^{\gamma-1}} \]
	which shows the uniform integrability of $V_x$.
\end{proof}

\begin{lemma} \label{lemma:f}
	$f$ is bounded on $[1,\infty)$.
\end{lemma}

\begin{proof}
	By assumption, we have $f(t) \to 1/a$ so we only have to show that $f$ is bounded on $[1,t]$ for every $t \geq 1$, which comes from the fact that $t f(t) = \E(\mu(t))$ is non-decreasing.
\end{proof}

\subsection{Convergence of $(R_p, \cZ_p)$} \label{sub:RZ}

We prove in this section that $(R_p, \cZ_p) \Rightarrow (a \cZ_\infty, \cZ_\infty)$. It is well-known that in order to prove the convergence of a random walk such as $(R_p, \cZ_p)$, one needs to control the law of its increments $(R(1), \cZ(1))$, and we begin by describing this law.

We enrich the probability space so as to consider a triple $(V, \mu, U) \in (0,\infty) \times \cM \times [0,1]$ such that under $\P$, $U$ is independent from $(V, \mu)$ and is uniformly distributed on $[0,1]$. From this triplet $(V, \mu, U)$ we construct the following random variables:
\[ \widehat Z = [U \mu(V)] \ \text{ and } \ \widehat R = A(\widehat Z). \]
Thus, under $\P$, $\widehat Z$ is the rank of an atom of $\cP = \mu\mid_V$ chosen uniformly at random and $\widehat R$ is its location. We also consider the measure $\widehat \P$ which is absolutely continuous with respect to $\P$ with Radon--Nikodym derivative $\mu(V)$: for any random variable $X \geq 0$,
\[ \widehat \E(X) = \E \left( \mu(V) X \right). \]
The measure $\widehat \P$ is the usual \textit{offspring-biased} measure, see for instance~\cite{Lyons95:0, Lyons96:0}.

\begin{lemma} \label{lemma:law}
	$(V_{T(1)-1}, \cZ(1), R(1))$ under $\P$ is equal in distribution to $(V, \widehat Z, \widehat R)$ under $\widehat \P$. In particular, for any $f: \R_+ \times \N \times \N \to \R_+$ we have
	\[ \E \left[ f \left( V_{T(1)-1}, \cZ(1), R(1) \right) \right] = \E \left[ \mu(V) f(V, \widehat Z, \widehat R) \right]. \]
\end{lemma}

\begin{proof}
	Let $f: \R_+ \times \N \times \N \to \R_+$ be measurable. By definition,
	\begin{align*}
		\widehat \E(f(V, \widehat Z, \widehat R)) & = \E \left( \mu(V) f(V, \widehat Z, A(\widehat Z)) \right)\\
		& = \sum_{x \geq 1} x \E \left( f(V, [Ux], A([Ux])) ; \mu(V) = x \right).
	\end{align*}
	Since $U$ under $\P$ is uniformly distributed on $[0,1]$, $[Ux]$ is uniformly distributed on $\{0, \ldots, x-1\}$ and so
	\begin{align*}
		\widehat \E(f(V_{T(1)-1}, \widehat Z, \widehat R)) & = \sum_{x \geq 1} \sum_{z = 0}^{x-1} \E \left( f(V, z, A(z)) ; \mu(V) = x \right)\\
		& = \sum_{x \geq 1, z \geq 0} \E \left( f(V, z, A(z)) ; \mu(V) = x+z \right).
	\end{align*}
	
	On the other hand, we have according to~\cite[Equation (2.3)]{Schertzer19:0}
	\[ \E(f(\cZ(1), R(1))) = \sum_{x \geq 1, z \geq 0} \E \left( f(z, A_\cP(z)) ; \lvert \cP \rvert = x+z \right) \]
	and the argument of the proof of this equality readily generalizes to give
	\[ \E(f(V_{T(1)-1}, \cZ(1), R(1))) = \sum_{x \geq 1, z \geq 0} \E \left( f(V, z, A_\cP(z)) ; \lvert \cP \rvert = x+z \right). \]
	Since $\lvert \cP \rvert = \mu(V)$ and $A(z) = A_{\cP}(z)$ for $z \leq \mu(V)$, the two expressions match, i.e., $\E(f(V_{T(1)-1}\cZ(1), R(1))) = \widehat \E(f(V, \widehat Z, \widehat R))$, which proves the result.
\end{proof}

This expression for the law of $(V_{T(1)-1}, R(1), \cZ(1))$ makes it possible to prove the two following technical lemmas.

\begin{lemma}\label{lemma:g>}
	Let $g: \R^2 \to \R$ be measurable, bounded and such that the Lebesgue measure of its set of discontinuity points is zero. Then for any $\delta > 0$ we have
	\[ p v_p \E \left( g\left(v_p R(1), v_p \cZ(1)\right) ; V_{T(1)-1} \geq \delta / v_p \right) \to \frac{1}{a \delta^{\gamma-1}} \E \left( V_\infty g \left( \delta U V_\infty, \frac{\delta U V_\infty}{a} \right) \right) \]
	where $U$ and $V_\infty$ are independent.
\end{lemma}

\begin{proof}
	Lemma~\ref{lemma:law} gives
	\[ p v_p \E \left( g(v_p R(1), v_p Z(1)) ; V_{T(1)-1} \geq \delta / v_p \right) = p v_p \E \left( \mu(V) g(v_p \widehat R, v_p \widehat Z) ; V \geq \delta / v_p \right). \]
	Let $x = \delta / v_p$: recalling that, by definition, $V_x$ is equal in distribution to $V/x$ conditioned on $V \geq x$ and that $\mu$ and $V$ are independent, we can write
	\begin{multline*}
		p v_p \E \left( \mu(V) g(v_p \widehat R, v_p \widehat Z) ; V \geq \delta / v_p \right)\\
		= p (\delta / x) \P(V \geq x) \E \left( \mu(x V_x) g( (\delta / x) A([U \mu(x V_x)]), (\delta / x) [U \mu(x V_x)]) \right)
	\end{multline*}
	with $\mu$ and $V_x$ independent. Introducing
	\[ \Phi(u,v) = \frac{\mu(uv)}{uv} g \left( \frac{\delta A([U \mu(uv)])}{u} , \frac{\delta [U \mu(uv)]}{u} \right) \ \text{ and } \ \phi(u,v) = \E(\Phi(u,v)), \]
	we can rewrite this as
	\begin{equation} \label{eq:g}
		p v_p \E \left( g(v_p R(1), v_p Z(1)) ; V_{T(1)-1} \geq \delta / v_p \right) = p \delta \P(V \geq \delta / v_p) \E \left( V_x \phi(x, V_x) \right).
	\end{equation}
	By choice of $v_p$ we have $p \P(V \geq \delta / v_p) \to \delta^{-\gamma}$ and so it remains to show that $\E(V_x \phi(x, V_x)) \to \E(V_\infty g(\delta UV_\infty, \delta U V_\infty / a)) / a$. Since $V_x \Rightarrow V_\infty$ and $V_x$ is uniformly integrable (Lemma~\ref{lemma:V_x}) and $\lvert \phi(u,v) \rvert \leq \sup \lvert g \rvert f(uv)$ with $f$ bounded (Lemma~\ref{lemma:f}), we only have to show that $\phi(x, \alpha(x)) \to a^{-1} \E [g(\delta U v_\infty, \delta U v_\infty / a)]$ as $x \to \infty$, for any function $\alpha$ with $\alpha(x) \to v_\infty \in [1,\infty)$ as $x \to \infty$.
	
	For such $\alpha$, we have as $x \to \infty$
	\[ \frac{\mu(x \alpha(x))}{x \alpha(x)} \Rightarrow \frac{1}{a} \ \text{ and } \ \left( \frac{\delta A([U \mu(x \alpha(x))])}{x} , \frac{\delta [U \mu(x \alpha(x))]}{x} \right) \Rightarrow (\delta U v_\infty, \delta U v_\infty / a). \]
	Since by assumption, $g$ is almost surely continuous at $(\delta U v_\infty, \delta U v_\infty / a)$, we obtain $\Phi(x,\alpha(x)) \to g(\delta U v_\infty, \delta U v_\infty / a) / a$.
	
	Next, since $g$ is bounded, we have $\Phi(x, \alpha(x)) \leq c \mu(x \alpha(x))/(x \alpha(x))$ for some finite constant $c > 0$. Since the bound $\mu(x \alpha(x)) / (x \alpha(x))$ converges to $1/a$ in $L_1$ and also in the mean, the dominated convergence theorem~\cite[Theorem 1.21]{Kallenberg02:0} implies that $\phi(x, \alpha(x)) \to \E [g(\delta U v_\infty, \delta U v_\infty / a)] / a$ as desired.
\end{proof}

\begin{lemma}\label{lemma:g<}
	Let $g: \R^2 \to \R$ be measurable, bounded and $0$ in a neighborhood of~$0$. Then for $\delta > 0$ small enough we have
	\[ p v_p \E \left( g\left(v_p R(1), v_p \cZ(1)\right) ; V_{T(1)-1} < \delta / v_p \right) \to 0. \]
\end{lemma}

\begin{proof}
	In the proof we will use the following identity: if $X \geq 0$ and $U$ is an independent uniform random variable, then for every $x \geq 0$ one readily checks that
	\begin{equation} \label{eq:id}
		\E(X ; UX \geq x) = \int_x^\infty \P(X \geq u) \d u.
	\end{equation}

	Let us now proceed with the proof. Let $c = \sup \lvert g \rvert$ and $m$ such that $g(u,v) = 0$ if $\lvert u \rvert < m$ and $\lvert v \rvert < m$. Fix in the rest of the proof $\delta > 0$ with $\delta < m$ and $\delta < am$. Then
	\begin{multline*}
		\left \lvert p v_p \E \left( g(v_p R(1), v_p \cZ(1)); V_{T(1)-1} < \delta/v_p \right) \right \rvert\\
		\leq c p v_p \E \left( 1(v_p R(1) > m) + 1(v_p \cZ(1) > m); V_{T(1)-1} < \delta/v_p \right)
	\end{multline*}
	and since $\delta < m$ and $R(1) \leq V_{T(1)-1}$, we have $\P(m/v_p < R(1), V_{T(1)-1} < \delta / v_p) = 0$ and so we obtain
	\begin{multline*}
		\left \lvert p v_p \E \left( g(v_p R(1), v_p \cZ(1)); V_{T(1)-1} < \delta/v_p \right) \right \rvert\\
		\leq c p v_p \P \left( v_p \cZ(1) > m, V_{T(1)-1} < \delta/v_p \right).
	\end{multline*}
	Using Lemma~\ref{lemma:law}, we obtain
	\begin{align*}
		p v_p \P \left( v_p \cZ(1) > m \right) & = p v_p \E\left( \mu(V) ; [U \mu(V)] \geq m / v_p \right)\\
		& = p v_p \E\left( \mu(V) ; U \mu(V) \geq [m / v_p] \right).
	\end{align*}
	According to Lemma~\ref{lemma:DA-mu(V)}, $\mu(V)$ is in the domain of attraction of a $\gamma$-stable distribution, so there exists $\ell$ slowly varying with $\P(\mu(V) \geq x) = \ell(x) / x^\gamma$. Then according to~\eqref{eq:id},
	\[ p v_p \P \left( v_p \cZ(1) > m \right) = p v_p \int_{[m/v_p]}^\infty \ell(u) u^{-\gamma} \d u \]
	and so Proposition $1.5.10$ in~\cite{Bingham89:0} implies that
	\[ p v_p \P \left( v_p \cZ(1) > m \right) \sim p v_p \frac{\ell(m/v_p)}{(\gamma-1) (m/v_p)^{\gamma-1}} \sim \frac{1}{(\gamma-1) a^\gamma m^{\gamma-1}} \]
	with the last equivalence coming from Lemma~\ref{lemma:DA-mu(V)} which implies that $p v^\gamma_p \ell(1/v_p) \to a^{-\gamma}$. On the other hand, taking $g$ equal to $1(z>m)$, Lemma~\ref{lemma:g>} gives
	\[ p v_p \P( v_p \cZ(1) > m, V_{T(1)-1} \geq \delta/v_p) \to \frac{1}{a \delta^{\gamma-1}} \E \left( V_\infty ; \delta U V_\infty \geq a m \right). \]
	Using~\eqref{eq:id}, we obtain
	\begin{align*}
		\delta^{1-\gamma} a^{-1} \E (V_\infty ; \delta U V_\infty \geq m) & = \frac{1}{a \delta^{\gamma-1}} \int_{am/\delta}^\infty \P(V_\infty \geq u) \d u = \frac{1}{(\gamma-1) a^\gamma m^{\gamma-1}}
	\end{align*}
	using for the last inequality the fact that $am/\delta > 1$ and $\P(V_\infty \geq x) = x^{-\gamma}$ for $x \geq 1$. For $\delta$ small enough, we have proved that $p v_p \P( v_p \cZ(1) > m, V_{T(1)-1} \geq \delta/v_p)$ and $p v_p \P( v_p \cZ(1) > m)$ have the same limit. Thus, their difference vanishes which is exactly the desired result.
\end{proof}

\begin{corollary} \label{lemma:R-Z}
	For any $x, y > 0$ we have
	\[ p v_p \P \left( v_p R(1) \geq x, v_p \cZ(1) \geq y \right) \to \frac{1}{a (\gamma-1) (\max(x, ay))^{\gamma-1}}. \]
	In particular, $(R_p, \cZ_p) \Rightarrow (a \cZ_\infty, \cZ_\infty)$.
\end{corollary}

\begin{proof}
	Consider $\delta > 0$ and write
	\begin{multline*}
		p v_p \P \left( v_p R(1) \geq x, v_p \cZ(1) \geq y \right) = p v_p \P \left( v_p R(1) \geq x, v_p \cZ(1) \geq y, V_{T(1) - 1} \geq \delta / v_p \right)\\
		+ p v_p \P \left( v_p R(1) \geq x, v_p \cZ(1) \geq y, V_{T(1) - 1} < \delta / v_p \right).
	\end{multline*}
	For $\delta$ small enough, the second term vanishes by Lemma~\ref{lemma:g<}, so that we obtain by Lemma~\ref{lemma:g>}
	\[ p v_p \P \left( v_p R(1) \geq x, v_p \cZ(1) \geq y \right) \to \frac{1}{a \delta^{\gamma-1}} \E \left( V_\infty ; \delta U V_\infty \geq x, \delta U V_\infty \geq a y \right). \]
	Let $d = \max(x, ay)$: then for $\delta < d$, we have by definition of $V_\infty$ and~\eqref{eq:id}
	\[ \E \left( V_\infty ; \delta U V_\infty \geq x, \delta U V_\infty \geq a y \right) = \int_{d/\delta}^\infty u^{-\gamma} \d u = \frac{(\delta/d)^{\gamma-1}}{\gamma-1} \]
	which gives the desired convergence. This convergence implies that $(R_p(1), \cZ_p(1))$ converges in distribution to a bivariate stable random variable by an application of a bivariate Tauberian theorem that generalizes the argument of Gnedenko and Kolmogorov  \cite{Gnedenko68:0} mentioned above. 
		One can for instance use the results in~\cite{Haan78:0}.
	
	It remains to identify the limit. The same arguments as above lead to
	\[ p v_p \P \left(v_p \cZ(1) \geq y \right) \to \frac{1}{a^\gamma (\gamma-1) y^{\gamma-1}} \]
	and so the limit of $p v_p \P \left( v_p R(1) \geq x, v_p \cZ(1) \geq y \right)$ is the same as the limit obtained by replacing $R(1)$ by $a \cZ(1)$. This shows that $(R_p(1), \cZ_p(1)) \Rightarrow (a \cZ_\infty(1), \cZ_\infty(1))$, and this convergence implies the desired result.
\end{proof}

\begin{remark} \label{rk:VR}
	With the same arguments as in the previous proof, one can show that $p v_p \P \left( v_p R(1) \geq x \right) \to 1 / (a (\gamma-1) x^{\gamma-1})$. In particular, $R(1)$ is in the domain of attraction of stable distribution with index $\gamma-1 \in (0,1)$, so that $\E(R(1)) = \infty$. This proves the claim made in the introduction that we focus in the present paper in the case $\E(V) < \infty$ and $\E(\int u \cP(\d u)) = \infty$ (since $\E(\int u \cP(\d u)) = \E(R(1))$).
\end{remark}

\subsection{Proof of Theorem~\ref{thm:main}} \label{sub:proof}

We now proceed to the proof of Theorem~\ref{thm:main}, i.e., we assume that Assumption A holds and we prove that
\begin{equation} \label{eq:conv-thm}
	(\H_p, \C_p, \underline S_p) \fdd \ \left( a \underline S_\infty, a \underline S_\infty(\cdot/2\E(V)), \underline S_\infty \right).
\end{equation}

\subsubsection{It is enough to prove that $(\H_p, \underline S_p) \fdd \ (a \underline S_\infty, \underline S_\infty)$}

First, we recall Theorem~$4.1$ in~\cite{Schertzer18:0} which, in the non-triangular case of the present paper, can be stated as follows.

\begin{theorem}[Theorem $4.1$ in~\cite{Schertzer18:0}]
	If the following conditions hold:
	\begin{itemize}
		\item $\E(V) < \infty$ and $\E(\mu(V)) = 1$;
		\item $S_p \Rightarrow S_\infty$;
		\item $\H_p \fdd \H_\infty$ for some process $\H_\infty$ which is (almost surely) continuous at $0$ and satisfies the condition $\P(\H_\infty(t) > 0)$ for every $t > 0$;
	\end{itemize}
	then $(\H_p, \C_p) \fdd (\H_\infty, \H_\infty(\, \cdot \,/(2\E(V)))).$
\end{theorem}

As the two first assumptions of this result hold under Assumption A (see Lemma~\ref{lemma:DA-mu(V)} for the convergence $S_p \Rightarrow S_\infty$), it follows from this result that in order to prove~\eqref{eq:conv-thm} it is enough to prove that $(\H_p, \underline S_p) \fdd \ (a \underline S_\infty, \underline S_\infty)$.
%
%

\subsubsection{Proof of $(\H_p, \underline S_p) \fdd \ (a \underline S_\infty, \underline S_\infty)$}

By definition of tightness and of the product topology, tightness of two sequences $(X^1_p)$ and $(X^2_p)$ is equivalent to the tightness of the pair of sequences $(X^1_p, X^2_p)$ in the product topology. However, the following stronger result holds for random walks.

\begin{lemma} \label{lemma:sk}
Let $(X^1_p, X^2_p, X^3_p)$ be a three dimensional random walk. If each sequence $(X^1_p)$, $(X^2_p)$ and $(X^3_p)$ is tight in $D(\R)$, then $(X^1_p, X^2_p, X^3_p)$ is tight in $D(\R^3)$.
\end{lemma}

\begin{proof}
	The proof follows by the equivalence of the following statements:
	\begin{enumerate}
		\item $(X_p^1)$, $(X_p^2)$ and $(X_p^3)$ are each tight in $D(\R)$;
		\item $(X_p^1(1))$, $(X_p^2(1))$ and $(X_p^3(1))$ are each tight in $\R$;
		\item $(X_p^1(1), X_p^2(1), X_p^3(1))$ is tight in $\R^3$;
		\item $(X_p^1, X_p^2, X_p^3)$ is tight in $D(\R^3)$.
	\end{enumerate}
The implications (1) $\Rightarrow$ (2), (2) $\Rightarrow$ (3) and (3) $\Rightarrow$ (4) are immediate, while the implication (4) $\Rightarrow$ (1) holds true because $(X^1_p, X^2_p, X^3_p)$ is a random walk, see for instance~\cite[Corollary VII.3.6]{Jacod03:0}.
\end{proof}

Since $(T_p, \cZ_p) \Rightarrow (T_\infty, \cZ_\infty)$ by Corollary~\ref{lemma:S-T-Z} and $(R_p, \cZ_p) \Rightarrow (a \cZ_\infty, \cZ_\infty)$ by Corollary~\ref{lemma:R-Z}, it follows that $(T_p, \cZ_p, R_p) \Rightarrow (T_\infty, \cZ_\infty, a \cZ_\infty)$, the convergence taking place in $D(\R^3)$ by Lemma~\ref{lemma:sk}. We now show that the convergence of this trivariate renewal process implies that $(\H_p, \underline S_p) \fdd \ (a \underline S_\infty, \underline S_\infty)$.


By the Skorohod embedding theorem, we can assume without loss of generality that $(R_p, \cZ_p, T_p)$ converges to $(a\cZ_\infty, \cZ_\infty, T_\infty)$ almost surely. Fix in the rest of the proof some $t \geq 0$. For $p \in \N \cup \{\infty\}$ let $\overline S_p(t) = \sup_{[0,t]} S_p - S_p(t)$ be the process $S_p$ reflected above its past supremum and $L_p$ its local time process at $0$, which is the right-continuous inverse of $T_p$:
\[ L_p(t) = \inf \left\{ s \geq 0: T_p(s) > t \right\}. \]

First, note that because $S_\infty$ is spectrally positive, is not a subordinator or a pure drift and has unbounded variation, $0$ is regular for $(-\infty, 0)$ according to~\cite[Theorem VII.1]{Bertoin96:0} and also for $(0,\infty)$ according to~\cite[Corollary VII.5]{Bertoin96:0}.

Since $S_\infty$ is spectrally positive, we have $\P(\overline S_\infty(t) = 0) = \P(\inf_{[0,t]} S_\infty = 0)$ by~\cite[Proposition VI.3]{Bertoin96:0}. We deduce that $\overline S_\infty(t) > 0$ almost surely. Since $\Delta T_\infty(L_\infty(t))$ corresponds to the size of the excursion of $\overline S_\infty$ away from $0$ straddling $t$~\cite[Proposition IV.5]{Bertoin96:0}, an immediate consequence of $\overline S_\infty(t) > 0$ is that $T_\infty$ jumps at time $L_\infty(t)$.

Also, this implies that $T_p \circ L_p(t) \to T_\infty \circ L_\infty(t)$. Indeed, for $p \in \N \cup {\infty}$ we have
\[ T_p(L_p(t)) = \inf\{s > t: \overline S_p(s) = 0\}. \]
(For $p \in \N$ this comes from the various definitions, and for $p = \infty$ see for instance~\cite[Proposition IV.7]{Bertoin96:0}.)
A a consequence, $T_p(L_p(t))-t$ is the hitting time of $0$ by $\overline S_p(t \, + \, \cdot)$. Because the reflection does not kick in before $\overline S_\infty$ hits $0$, conditionally on $\overline S_\infty(t) = x$ (which is almost surely $>0$ by the previous remark) this is the same in distribution as the entrance time in $[0,\infty)$ of $S_\infty$ started in $-x$. Since $0$ is regular for $(0,\infty)$, this entrance time is continuous with respect to the Skorohod topology at $S_\infty$, see for instance~\cite[Proposition VII.2.11]{Jacod03:0} or~\cite[Lemma 3.1]{Lambert15:0}. We thus obtain $T_p(L_p(t)) \to T_\infty(L_\infty(t))$. Since $t - T_p(L_p(t)-)$ is the hitting time of $S_p$ looked backward in time, the same argument shows by reversing time that $T_p(L_p(t)-) \to T_\infty(L_\infty(t)-)$ and so $\Delta T_p(L_p(t)) \to \Delta T_\infty(L_\infty(t))$.

Since $\Delta T_\infty(L_\infty(t)) \neq 0$, $T_p \to T_\infty$ in $D(\R)$ and $L_p(t) \to L_\infty(t)$, the convergence $\Delta T_p(L_p(t)) \to \Delta T_\infty (L_\infty(t))$ means that the jump of $T_\infty$ at $L_\infty(t)$ is carried in the pre-limit at time $L_p(t)${, see for instance~\cite[Proposition VI.2.1]{Jacod03:0}}. Since $(R_p, \cZ_p, T_p) \to (a \cZ_\infty, \cZ_\infty, T_\infty)$ in $D(\R^3)$, this implies that
\begin{equation} \label{eq:conv-R-Z}
	(R_p, \cZ_p) \big( L_p(t)- \big) \to (a \cZ_\infty, \cZ_\infty) \big( L_\infty(t)- \big)
\end{equation}
by standard properties of the Skorohod topology, see for instance~\cite[Proposition VI.2.1]{Jacod03:0}.

By Theorem \ref{thm:recap-2}, there exists some sequence $\rho_p \to \infty$ such that for each $t \geq 0$ we have
\[ (\H_p(t), \underline S_p(t)) = (R_{\rho_p}, \cZ_{\rho_p}) \big( L_{\rho_p}(t) - \big), \ \ \mbox{in law}. \]
In view of this equality in distribution, the convergence~\eqref{eq:conv-R-Z} therefore implies that $(\H_p(t), \underline S_p(t)) \Rightarrow (a \cZ_\infty, \cZ_\infty) (L_\infty(t)-)$. Since on the other hand $\underline S_p \Rightarrow \underline S_\infty$ by Corollary~\ref{lemma:S-T-Z}, we see that $\underline S_\infty(t)$ and $\cZ_\infty(L_\infty(t)-)$ are equal in distribution, the previous convergence achieves the proof of Theorem~\ref{thm:main}.






\providecommand{\bysame}{\leavevmode\hbox to3em{\hrulefill}\thinspace}
\providecommand{\MR}{\relax\ifhmode\unskip\space\fi MR }
\providecommand{\MRhref}[2]{%
  \href{http://www.ams.org/mathscinet-getitem?mr=#1}{#2}
}
\providecommand{\href}[2]{#2}




\end{document}